\documentclass[11pt,reqno]{amsart}

\usepackage{marginnote,amssymb,mathrsfs,latexsym,verbatim,pdfsync,color,mathabx}

\newcommand{\NN}{\mathbb{N}}

\newcommand{\CZ}{Calder\'on--Zygmund }

\newcommand{\di}{\,{\rm{d}}}
\def\Re{\operatorname{Re}}%{\, \text{\rm Re\,}}
\def\Im{\operatorname{Im}}%{\, \text{\rm Im\,}}

\DeclareFontFamily{U}{mathx}{\hyphenchar\font45}
\DeclareFontShape{U}{mathx}{m}{n}{
      <5> <6> <7> <8> <9> <10>
      <10.95> <12> <14.4> <17.28> <20.74> <24.88>
      mathx10
      }{}
\DeclareSymbolFont{mathx}{U}{mathx}{m}{n}
\DeclareFontSubstitution{U}{mathx}{m}{n}
\DeclareMathAccent{\widecheck}{0}{mathx}{"71}
\DeclareMathAccent{\wideparen}{0}{mathx}{"75}
\makeatletter
\newcommand{\leqnomode}{\tagsleft@true}
\newcommand{\reqnomode}{\tagsleft@false}
\makeatother
\newtheorem{theorem}{Theorem}
\newtheorem{corollary}[theorem]{Corollary}
\newtheorem{proposition}[theorem]{Proposition}
\newtheorem{definition}[theorem]{Definition}

\newtheorem{remark}[theorem]{Remark}

\begin{document}

\title[BMO spaces on weighted homogeneous trees]{BMO spaces on weighted homogeneous trees}

\author[L.\ Arditti]{Laura Arditti}
\address{Dipartimento di Scienze Matematiche ``Giuseppe Luigi Lagrange'', Dipartimento di Eccellenza 2018-2022,  Politecnico di Torino,
		Corso Duca degli Abruzzi 24, 10129 Torino Italy }
\email{laura.arditti@polito.it}

\author[A.\ Tabacco]{Anita Tabacco}
\address{Dipartimento di Scienze Matematiche ``Giuseppe Luigi Lagrange'', Dipartimento di Eccellenza 2018-2022,  Politecnico di Torino,
		Corso Duca degli Abruzzi 24, 10129 Torino Italy }
\email{anita.tabacco@polito.it}

\author[M.\ Vallarino]{Maria Vallarino}
\address{Dipartimento di Scienze Matematiche ``Giuseppe Luigi Lagrange'', Dipartimento di Eccellenza 2018-2022,  Politecnico di Torino,
		Corso Duca degli Abruzzi 24, 10129 Torino Italy }
\email{maria.vallarino@polito.it}

\thanks{{\em Math Subject Classification} 05C05; 30H35; 42B30}

\begin{abstract}
We consider an infinite homogeneous tree $\mathcal V$ endowed with the usual metric $d$ defined on graphs and a weighted measure $\mu$. The metric measure space $(\mathcal V,d,\mu)$ is nondoubling and of exponential growth, hence the classical theory of Hardy and $BMO$ spaces does not apply in this setting. We introduce a space $BMO(\mu)$ on $(\mathcal V,d,\mu)$ and investigate some of its properties. We prove in particular that $BMO(\mu)$ can be identified with the dual of a Hardy space $H^1(\mu)$ introduced in a previous work and we investigate the sharp maximal function related with $BMO(\mu)$. 
\end{abstract}
\maketitle

\begin{centerline}
{\it Dedicated to Guido Weiss on the occasion of his 90th birthday} 
\end{centerline}

\section{Introduction}

The classical space of functions of bounded mean oscillation $BMO$ was introduced in the Euclidean setting by John and Nirenberg \cite{JN}. It is defined as the set of locally integrable functions $f$ such that 
\begin{equation}\label{defBMO}
\sup_{B}\frac{1}{|B|}  \int_B|f-f_B| \di x<\infty\,,
\end{equation}
where the supremum is taken over all Euclidean balls and $f_B$ denotes the average of $f$ on $B$. A celebrated result of Fefferman and Stein \cite{FeS} identifies $BMO$ with the dual of the classical Hardy space $H^1$. 

Various extensions of such theory have been considered in the literature. The first extension was developed 
on spaces of homogeneous type in the sense of Coifman and Weiss \cite{CW1, CW, S, SW}. These are metric measure spaces $(X,d,\mu)$ where the 
doubling condition is satisfied, i.e., there exists
a constant $C$ such that
\begin{equation}\label{doubling}
\mu\bigl(B(x,2r)\bigr)
\leq C\, \mu\bigl(B(x,r)\bigr)
\qquad\forall x \in X\,, \quad\forall r >0,
\end{equation}
where $B(x,r)$ denotes the ball centred at $x$ of radius $r$ in the metric $d$. In this setting functions in $BMO(\mu)$ satisfy the analogue of condition \eqref{defBMO}, where metric balls are involved. Subsequently extensions of the theory of Hardy and $BMO$ spaces have been considered 
in the literature on various metric measure spaces which do not satisfy the doubling condition (\ref{doubling}). Due to the lack of the doubling condition, it is less clear which is a suitable condition to define a $BMO$ space which enjoys all the properties of the classical one, and in particular it is not clear which subsets of the space can be chosen to replace balls in condition \eqref{defBMO}. 

The literature on this subject is huge and we shall only cite here some 
contributions. In particular, various results on this subject have been obtained on nondoubling Riemannian manifolds  \cite{CMM,  MMV, T} and on Lie groups of exponential growth \cite{MOV, V}. A few results have been obtained in the discrete setting of infinite graphs \cite{CM,Fe}. 
  
\smallskip

The goal of the present paper is to develop a theory for a $BMO$ space on a homogeneous tree $\mathcal V$ endowed with the usual metric $d$ defined on a graph and a weighted measure $\mu$ (see Section \ref{notation} for the details), such that $(\mathcal V,d,\mu)$ is a nondoubling space. Such weighted homogeneous trees were first studied by Hebisch and Steger \cite{HS}, who developed a \CZ theory on them.  In particular, they proved that there exists a family of appropriate sets in $\mathcal V$, which are called \CZ sets, which replace the family of balls in the classical \CZ theory. In \cite{ATV1} we introduced an atomic Hardy space $H^1(\mu)$ on $(\mathcal V,d,\mu)$, where atoms are functions supported in \CZ sets, with 
vanishing integral and satisfying a certain size condition. %We proved in \cite{ATV} that the Hardy space $H^1(\mu)$ has good interpolation properties with the $L^p(\mu)$-spaces and studied boundedness of singular integral operators on it. 

We shall define here a space of functions of bounded mean oscillation $BMO(\mu)$ adapted to this setting, for which the oscillation in the analogue of condition \eqref{defBMO} is measured on \CZ sets. Then we show that $BMO(\mu)$ can be identified with the dual of the Hardy space $H^1(\mu)$. More precisely, we introduce a family of spaces $BMO_q(\mu)$, with $q\in [1,\infty)$, for which the integrability condition \eqref{defBMO} is expressed in terms of an $L^q$-norm, and show that all such spaces coincide. As a consequence, we find the real interpolation spaces between $L^{r}(\mu)$ and $BMO(\mu)$, for $r\in [1,\infty)$. It would be interesting to find the complex interpolation space between $L^{r}(\mu)$ and $BMO(\mu)$, as well. To work in this direction, we introduce and study the sharp maximal function associated with \CZ sets, and show that the $L^p$-norm of a function is controlled by the $L^p$-norm of a variant of its sharp maximal function (see Theorem \ref{t: normaLpfsharp}). %AGGIUNGERE CHE NON SAPPIAMO FARE L'INTERPOLAZIONE COMPLESSA? 

%   
%
%In particular, we mention that Celotto and Meda \cite{CM} studied various Hardy spaces on a homogeneous tree $\mathcal V$ endowed with the metric $d$ and the counting measure, which is not the measure $\mu$ that we consider here. Their theory is useful to study the boundedness of singular integral operators related to the standard Laplacian defined on trees which is self-adjoint with respect to the counting measure and not to the measure $\mu$. The theory we develop here instead is useful to study singular integral operators related to a distinguished Laplacian self-adjoint on $L^2(\mu)$ (see Subsection \ref{SubSecSingInt}). 

%We mention that in \cite{MOV, V} the authors used the \CZ theory introduced by Hebisch and Steger in \cite{HS} to construct Hardy and $BMO$ spaces on some solvable Lie groups of exponential growth and studied their properties. Our work can be thought as a counterpart in a discrete setting of the results in \cite{V}, and some of our proofs are strongly inspired by it. 

\bigskip

The paper is organized as follows. In Section \ref{notation} we introduce weighted homogeneous trees, we recall the definition of \CZ sets and study some of their geometric properties. In Section \ref{Hardy} we recall the definition of the Hardy space $H^1(\mu)$, we define the space $BMO(\mu)$ and prove the duality between these two spaces. As a consequence, we deduce a real interpolation result and a boundedness result for integral operators whose kernel satisfy a suitable H\"ormander condition. The last section is devoted to some inequalities involving the sharp maximal function defined in terms of \CZ sets. 

 \bigskip

Positive constants are denoted by $C$; these may differ from one line to another, and may depend on any quantifiers written, implicitly or explicitly, before the relevant formula.

	\section{Weighted homogeneous trees and \CZ sets} \label{notation}
	In this section we first introduce the infinite homogeneous tree and we define a distance and a measure on it. %We then recall the definition of \CZ sets. For the details on this subject we refer the reader to \cite{ATV1, HS}. %and prove some new geometric properties of such sets. 

	%The corresponding metric measure space $(\mathcal{V}, d, \mu)$ does not satisfy the doubling property. %We then introduce a family of sets, called trapezoids, which will be fundamental in the construction of Hardy spaces.
	\begin{definition}
		An infinite homogeneous tree of order $m+1$ is a graph $T=(\mathcal{V},\mathcal{E})$, where $\mathcal V$ denotes the set of vertices and $\mathcal E$ denotes the set of edges, with the following properties:
		\begin{enumerate}
			\item[(i)] $T$ is connected and acyclic;
			\item[(ii)] each vertex has exactly $m+1$ neighbours.
		\end{enumerate}
	\end{definition}	 
	
%	\begin{figure} [h]
%		\centering
%		\includegraphics[width=0.4\linewidth]{pics/tree}
%		\caption{A representation of the infinite homogeneous tree of order $3$.}
%		\label{fig:tree}
%	\end{figure}
	
	On $\mathcal{V}$  we define the distance $d(x,y)$ between two vertices $x$ and $y$ as the length of the shortest path between $x$ and $y$.
	We also fix a doubly-infinite geodesic $g$ in $T$, that is a connected subset $g \subset \mathcal{V}$ such that
	\begin{enumerate}
		\item[(i)] for each element $v \in g$ there are exactly two neighbours of $v$ in $g$;
		\item[(ii)] for every couple $(u,v)$ of elements in $g$, the shortest path joining $u$ and $v$ is contained in $g$.
	\end{enumerate}
	We define a mapping $N: g \rightarrow \mathbb{Z}$ such that 
	\begin{equation}
		\left| N(x) - N(y) \right|  = d(x,y) \qquad \forall x,y \in g \,.
	\end{equation}		
This corresponds to the choice of an origin $o \in g$ (the only vertex for which $N(o)=0$) and an orientation for $g$; in this way we obtain a numeration of the vertices in $g$.
	We define the level function $\ell: \mathcal{V} \rightarrow \mathbb{Z}$ as
	\[
	\ell(x) = N(x') - d(x,x')\,,
	\]
	where $x'$ is the only vertex in $g$ such that $d(x,x') = \min\lbrace d(x,z): z\in g \rbrace$.
	For $x,y \in \mathcal{V}$ we say that $y$ \textit{lies above} $x$ if
	\[
	\ell(x) = \ell(y) - d(x,y).
	\]
	In this case we also say that $x$ \textit{lies below} $y$.
	
%	\begin{figure}[!tbp]
%		\label{fig:pic1}
%		\begin{center}		\includegraphics[width=0.9\linewidth]{pic5_presentazione}
%			\caption{Representation of the measure $\mu$ ($q=3$)}
%		\end{center}
%	\end{figure}

	Let $\mu$ be the measure on $\mathcal{V}$ such that for each function $f: \mathcal{V}\rightarrow\mathbb C$
	\begin{equation}
		\int_{\mathcal{V}} f \, \di\mu = \sum_{x \in \mathcal{V}} f(x) m^{\ell(x)}.
	\end{equation}
	Then $\mu$ is a weighted counting measure such that the weight of a vertex depends only on its level, and the weight associated to a certain level is given by $q$ times the weight of the level immediately underneath. In particular, it can be proved \cite{ATV1} that for every $x_0\in\mathcal V$ and $r>0$ the measure of the ball centred at $x_0$ of radius $r$ is $\mu\big(B(x_0,r)\big) =m^{\ell(x_0)} \, \frac{m^{r+1}+m^r-2}{m-1}$. Hence, the metric measure space $(\mathcal V, d, \mu)$ is of exponential growth and nondoubling.

To develop a Calder\'on--Zygmund theory on this nondoubling setting, it is useful to introduce suitable subsets of $\mathcal V$, called trapezoids. These sets were first defined in \cite{HS}. We shall recall below their definition and their  properties.%, for whose proof we refer the reader to \cite{ATV1}.   
	%In this subsection we introduce the notion of trapezoid and recall the definition and the main properties of the admissible trapezoids introduced in \cite{HS}.
	
	\begin{definition}
		We call trapezoid a set of vertices $S \subset \mathcal V$ for which there exist $x_{S} \in \mathcal V$ and $a,b \in \mathbb R_+$ such that
		\begin{equation}
		S= \left\lbrace x \in \mathcal{V} : x \text{ lies below } x_S \,, a \leq \ell(x_S)-\ell(x) < b \right\rbrace.
		\end{equation}
	\end{definition}
	In the following we will refer to $x_S$ as the root node of the trapezoid.
	Among all trapezoids we are mostly interested in those where $a$ and $b$ are related by particular conditions, as specified in the following definitions.
	
	\begin{definition} \label{DefAdmissibleTrapezoid}
		A trapezoid $R \subset \mathcal{V}$ is an admissible trapezoid if and only if one of the following occurs:
		\begin{enumerate}
			\item[(i)] $R=\left\lbrace x_R \right\rbrace$ with $x_R \in \mathcal{V}$, that is $R$ consists of a single vertex;
			\item[(ii)] $\exists x_R \in \mathcal{V}, \, \exists h(R) \in \mathbb{N}^+$ such that
			\[
			R= \left\lbrace x \in \mathcal{V} : x \text{ lies below } x_R \,, h(R) \leq \ell(x_R)-\ell(x) < 2h(R) \right\rbrace.
			\]
		\end{enumerate}
		$R$ is called degenerate in case (i) and non-degenerate in case (ii). 
	\end{definition}
	We set $h(R)=1$ in the degenerate case. In both cases, $h(R)$ can be interpreted as the height of the admissible trapezoid, which coincides with the number of levels spanned by $R$. We call \emph{width} of the admissible trapezoid $R$ the quantity
		$
		w(R) = m^{\ell(x_R)}.
		$ We have that
	\begin{equation} \label{PropMeasureOfTrapezoid}
		\mu(R) = h(R)m^{\ell(x_R)} = h(R)w(R).
	\end{equation}
%	
%	\begin{figure}[!btp]
%		\centering
%		\includegraphics[width=1.1\linewidth]{pics/pic11_presentazione}
%		\caption{Representation of an admissible trapezoid with $h(R) = 2$ ($q=3$)}
%		\label{fig:pic2}
%	\end{figure}
%	
	
%	\begin{figure}[!btp]
%		\begin{center}		\includegraphics[width=1.1\linewidth]{pic11_presentazione} 
%			\label{fig:pic2}
%			\caption{Representation of an admissible trapezoid with $h(R) = 2$ ($q=3$)}
%		\end{center}
%	\end{figure}
%
%	\begin{proposition}\label{inclusion}
%		
%		Let $R_1$ and $R_2$ be two admissible trapezoids.
%		If
%		\[
%		R_1 \cap R_2 \neq \emptyset \text{   and   } w(R_1)\geq w(R_2) \,,
%		\]
%		then
%		 $
%		R_2 \subset \tilde{R_1}.
%		$
%	\end{proposition}
%	CONTROLLARE SE LA USIAMO!!!

We now introduce the family of Calder\'on--Zygmund sets. They are trapezoids, even if not of admissible type (except for the degenerate case); they consist of suitable enlargements of admissible trapezoids, constructed according to the following definition.

\begin{definition}
	Given a non-degenerate admissible trapezoid $R$, the envelope of $R$ is the set
	\begin{equation}
	\tilde{R}= \left\lbrace x \in \mathcal{V} : x \text{ lies below } x_R \,, \frac{h(R)}{2} \leq \ell(x_R)-\ell(x) < 4h(R) \right\rbrace\,;
	\end{equation}
	we set $h(\tilde{R})=h(R)$. The envelope of a non-degenerate admissible trapezoid is also called a non-degenerate Calder\'on--Zygmund set. Given a degenerate admissible trapezoid $R$, the envelope of $R$ is the set $\tilde R=R$. We denote by $\mathcal R$ the family of all the Calder\'on--Zygmund sets. 
	\end{definition}
	
	We refer the reader to \cite{ATV1} for the properties of such sets, in particular see \cite[Propositions 2, 4]{ATV1} for the proof of the following result. 

%	
%	
%	\begin{definition} A set $\tilde{Q} \subset \mathcal{V}$ is a Calder\'on--Zygmund set if and only if one of the following occurs:
%		\begin{enumerate}
%			\item[(i)] $\tilde{Q}=\left\lbrace x_{\tilde{Q}}
%			\right\rbrace$ with $x_{\tilde{Q}} \in \mathcal{V}$, that is $\tilde{Q}$ consists of a single vertex;
%			\item[(ii)] $\exists x_{\tilde{Q}}\in \mathcal{V}, \, \exists h \in \mathbb{N}^+$ such that
%			\[
%			\tilde{Q}= \left\lbrace x \in \mathcal{V} : x \text{ lies below } x_{\tilde{Q}} \,, \frac{h}{2} \leq \ell(x_{\tilde{Q}})-\ell(x) < 4h \right\rbrace.
%			\]
%		\end{enumerate} 
%	\end{definition}
%We set $h(\tilde Q)=1$ in the first case and $h(\tilde Q)=h$ in the second case. 
%	
%	Calder\'on--Zygmund sets are trapezoids, even if not of admissible type. They consist of suitable enlargements of admissible trapezoids, constructed according to the following definition.
%	
%	\begin{definition}
%		Given an admissible trapezoid $R$, we denote by $\tilde{R}$ the Calder\'on--Zygmund set such that:
%		\begin{enumerate}
%			\item[(i)] if $R=\{x_R\}$, then $\tilde{R}=R$;
%			\item[(ii)]  otherwise, $x_{\tilde{R}}=x_R$ and $h(\tilde{R}) = h(R)$.
%		\end{enumerate}
%	We say that $\tilde{R}$ is the envelope of the trapezoid $R$.
%	\end{definition}

	\begin{proposition}\label{mutildeR}
		Let $R$ be an admissible trapezoid and $\tilde R$ its envelope. Then
		\begin{itemize}
\item[(i)] $		 
		\mu(\tilde{R}) \leq 4\mu(R)\,;$
		\item [(ii)]  for all $z \in \tilde{R}$, we have 
		$\tilde{R} \subset B(z, 8h(\tilde{R})) \,.$
\end{itemize}
	\end{proposition}

For any Calder\'on--Zygmund set $\tilde R$ we define an enlargement of it, whose measure is comparable with its measure. This can be thought as a substitute for the doubling condition in this setting. 
	\begin{definition}\label{tildeQ*}
		Given a Calder\'on--Zygmund set $\tilde{R}$, we define the set
		\begin{equation}
			\tilde{R}^* = \left\lbrace x \in \mathcal{V} : d(x,\tilde{R}) < h({\tilde{R})}/4 \right\rbrace\,.
		\end{equation}
			\end{definition}
It is easy to see that there exists a positive constant $C$ such that for every Calder\'on--Zygmund set $\tilde{R}$
\begin{equation}\label{mutildeQ*}
\mu(\tilde{R}^*)\leq C\mu(\tilde{R})\,.
\end{equation}	

\smallskip

In the following proposition we construct a covering of $\mathcal V$ made by an increasing family of Calder\'on--Zygmund sets. 
 
 \begin{proposition}\label{l: ricoprimento}
%LEMMA SULLA FAMIGLIA CRESCENTE $\tilde R_n$ DI INSIEMI CZ CHE COPRONO $\mathcal V$. 

There exists a family of Calder\'on--Zygmund sets $\{\tilde{Q}_n\}_{n=0}^\infty$ such that $\tilde{Q}_n \subset \tilde{Q}_{n+1}$ and $\bigcup_n \tilde{Q}_n = \mathcal{V}$.
\end{proposition}

\begin{proof}
	Consider the family $\{\tilde Q_n\}_{n=0}^\infty$ where
	\begin{itemize}
		\item $\tilde{Q}_0$ is the Calder\'on--Zygmund set with root node $x_0=o$ and height $h_0=1$ (where $o$ denotes the only vertex in the doubly-infinite geodesic $g$ such that $\ell(o)=0$);
		\item $\forall n \geq 1$, $\tilde{Q}_n$ is the Calder\'on--Zygmund set with root node $x_n$ that is the father node of $x_{n-1}$, i.e. it is the only neighbour of $x_{n-1}$ that lies above $x_{n-1}$, and height $h_n=h_{n-1}+1$ (then we have $h_n=n+1$).
	\end{itemize}
	We first show that $\tilde{Q}_n \subset \tilde{Q}_{n+1}$.
	Let $x\in \tilde{Q}_n$. By definition $x$ lies below $x_n$, then by construction $x$ also lies below $x_{n+1}$. Moreover, $\ell(x_{n+1})-\ell(x) = \ell(x_n)+1-\ell(x)$ and we have that
	\begin{gather*}
	\ell(x_{n+1})-\ell(x) < 4h_n+1 < 4h_{n+1} \,, \qquad 
	\ell(x_{n+1})-\ell(x) \geq \frac{h_n}{2}+1 \geq \frac{h_{n+1}}{2} \,.
	\end{gather*}
	So we conclude that $x \in \tilde{Q}_{n+1}$.
	
	To show that $\bigcup_n \tilde{Q}_n = \mathcal{V}$, consider $x \in \mathcal{V}$. Denote by $k$ the smallest index such that $x$ lies below $x_k$ (and so $x$ lies below $x_j$, $\forall j \geq k$) and set $\ell(x_k)-\ell(x)=d$.
	We seek for an index $j\geq k$ such that $x \in \tilde{Q}_j$, that is
	\begin{equation*}
	\frac{j+1}{2} = \frac{h_j}{2} \leq \ell(x_j)-\ell(x) < 4h_j = 4(j+1) \,.
	\end{equation*}
	Observe that $\ell(x_j) - \ell(x) = \ell(x_j)-\ell(x_k)+\ell(x_k)-\ell(x) = j-k+d$, so that
	\begin{equation*}
	\frac{j+1}{2} \leq j-k+d <4(j+1) \quad \iff \begin{cases}
	j > \frac{d-k-4}{3}\\
	j \geq 1+2(k-d)
	\end{cases}
	\end{equation*}
	Thus it is sufficient to take $j \geq \max \left\lbrace  k, 1+\frac{d-k-4}{3}, 1+2(k-d) \right \rbrace$. 
\end{proof}

For every $p\in (1,\infty)$ and for every $n\in\mathbb N$, let $X_n^p$ be the space $L^p_0(\tilde Q_n)$ of all functions in $L^p(\mu)$ which are
supported in the set $\tilde Q_n$ introduced in Proposition \ref{l: ricoprimento} and have vanishing integral. The space $(X_n^p, \|\cdot\|_{L^p})$ is a Banach
space. We denote by $X^p$ the space $L^p_{c,0}(\mu)$ of all functions in $L^p(\mu)$ with compact support
and vanishing integral, interpreted as the strict inductive limit of the spaces $X_n^p$ 
(see \cite[II, p. 33]{Bou} for the definition of the strict inductive limit topology). This space will be a key ingredient of next section. In particular, we shall use the following fact.

\begin{proposition}\label{dualXp}
Let $p\in (1,\infty)$. For every function $F:\mathcal V\rightarrow \mathbb C$, the functional defined by 
\begin{equation}\label{formuladualXp}
\ell(g)=\int Fg\di\mu=\int (F+c)g\di\mu \qquad \forall g\in X^p, c\in\mathbb C,
\end{equation}
lies in the dual of $X^p$.% , and 
%\begin{equation}\label{estimateF}
%\Big|\int Fg\di\mu\Big|\leq \|F\|_{L^q(\tilde Q_n)}\,\|g\|_{L^p(\tilde Q_n)}\qquad \forall g\in X_n^p\,.
%\end{equation}
On the other hand, for every functional $\ell$ in the dual of $X^p$ there exists $F:\mathcal V\rightarrow \mathbb C$ such that \eqref{formuladualXp} holds.
\end{proposition}
\begin{proof}
On one hand, let $F:\mathcal V\rightarrow \mathbb C$ and consider the linear functional $\ell:X^p\rightarrow \mathbb C$ defined by $\ell(g)=\int Fg\di\mu$. Then for every $n\in\mathbb N$ by H\"older's inequality we have 
$$
\Big|\int Fg\di\mu\Big|\leq \|F\|_{L^q(\tilde Q_n)}\,\|g\|_{L^p(\tilde Q_n)}\qquad \forall g\in X_n^p\,,
$$
where $q=p'$. Hence, $\ell$ is continuous on every $X_n^p$ and continuous on $X^p$. 

Suppose now that $\ell$ is a continuous linear functional on $X^p$. Then, for every $n\in\mathbb N$, $\ell\in (X_n^p)^*$, hence it can be extended to a bounded linear functional on $L^p(\tilde Q_n)$ and there exists a function $F_n\in L^q(\tilde Q_n)$ such that 
$$
\ell(g)=\int (F_n+c) g\di\mu\qquad \forall g\in X_n^p, \quad c\in\mathbb C\,.
$$
Notice that we used the fact that $g$ has vanishing integral in the formula above. For every $n\in \mathbb N$ we choose the constant $c_n$ such that $\int_{\tilde Q_1} (F_n+c_n)\di\mu=0$. This implies that the restriction of $F_n+c_n$ on $\tilde Q_j$ coincide with $F_j+c_j$ for every $j<n$. Hence for every $x\in \mathcal V$ we can define 
$$ 
F(x)=F_n(x)+c_n\,,
$$
where $n$ is any integer such that $x\in\tilde Q_n$. Then 
$$
\ell(g)=\int (F+c)\,g\di\mu\qquad \forall g\in X^p,\quad c\in\mathbb C\,,
$$
as required.
\end{proof}

	\section{Hardy and $BMO$ spaces}\label{Hardy}
In this section we first recall the definition of atomic Hardy spaces given in \cite{ATV1}.

\begin{definition}
A function $a$ is a {\emph{$(1,p)$-atom}}, for $ p\in(1, \infty]$, 
if it satisfies the following properties:
\begin{enumerate}
\item[(i)] \,$a$ is supported in a \CZ set $\tilde R$;
\item[(ii)] \,\,$\|a\|_{L^p}\leq \mu (\tilde R)^{1/p-1}\,;$ 
\item[(iii)] \,\,$\int_{\mathcal V} a\di\mu =0$\,.
\end{enumerate}
\end{definition}
%Observe that a $(1,p)$-atom is in $L^1(\mu)$ and it is normalized 
%in such a way that its $L^1$-norm does not exceed $1$. 

\begin{definition}
The Hardy space $H^{1,p}(\mu)$ is the space of all functions $g$ in $ L^1(\mu)$ 
such that $g=\sum_{j\in\mathbb N} \lambda_j\, a_j$, where $a_j$ are $(1,p)$-atoms and $\lambda _j$ 
are complex numbers such that $\sum _{j\in\mathbb N} |\lambda _j|<\infty$. We denote by $\|g\|_{H^{1,p}}$ 
the infimum of $\sum_{j\in\mathbb N}|\lambda_j|$ over all decompositions $g=\sum_{j\in\mathbb N}\lambda_j\,a_j$, 
where $a_j$ are $(1,p)$-atoms. 
\end{definition}
The space $H^{1,p}(\mu)$ endowed with the norm $\|\cdot\|_{H^{1,p}}$ is a Banach space. 

For every $p\in (1,\infty]$ we also introduce the spaces  
$$H^{1,p}_{\rm{fin}}(\mu)=\Big\{g\in L^1(\mu):~g=\sum_{j=1}^N\lambda_j\,a_j,\,a_j~ (1,p)-atoms, \lambda_j\in\mathbb C, N\in\NN \Big\}\,.$$  
 
\begin{proposition}\label{coincidono}
For any $p\in (1,\infty)$, the following hold:
\begin{itemize}
\item[(i)] $H^{1,p}(\mu)=H^{1,\infty}(\mu)$ and there exists a constant $C_p$ such that 
$$
\|g\|_{H^{1,p}}\leq \|g\|_{H^{1,\infty}}\leq C_p\|g\|_{H^{1,p}}\,;
$$
\item[(ii)] for every $L\in (H^1(\mu))^*$, $\|L\|_{(H^{1,p})^*}\leq C_p\|L\|_{(H^1)^*}$;
\item[(iii)] $H^{1,\infty}_{\rm{fin}}(\mu)\subset H^{1,p}_{\rm{fin}}(\mu)$;
\item[(iv)] for every Calder\'on--Zygmund set $\tilde R$, $L^p_0(\tilde R)\subset H^{1,\infty}_{\rm{fin}}(\mu)$.
\end{itemize}
 \end{proposition}
 \begin{proof}
 Property (i) follows from \cite[Proposition 5]{ATV1}. Take now $L\in (H^1(\mu))^*$ and $g\in H^1(\mu)$. Then
 $$
 |L(g)|\leq \|L\|_{(H^1)^*}\|g\|_{H^{1,\infty}}\leq  C_p\|L\|_{(H^1)^*}\|g\|_{H^{1,p}}\,,
 $$
 so that (ii) follows. %$\|L\|_{(H^{1,p})^*}\leq C_p\|L\|_{(H^1)^*}$.
 
 Property (iii) holds since every $(1,\infty)$-atom is a $(1,p)$-atom. 
 
To prove (iv) let $\tilde R$ be a Calder\'on--Zygmund set and $g\in L^p_0(\tilde R)$ be a function in $L^p(\mu)$ supported in $\tilde R$, with vanishing integral. Then $\|g\|_{L^{\infty}}=\max_{x\in \tilde R}|g(x)|<\infty$. Hence, $\mu(\tilde R)^{-1}\|g\|_{L^{\infty}}^{-1}g$ is a $(1,\infty)$-atom. This proves (iv). 
 \end{proof}
In the sequel we shall denote by $H^1(\mu)$ the space $H^{1,\infty}(\mu)$ endowed with the norm $\|\cdot\|_{H^1}=\|\cdot\|_{H^{1,\infty}}$.

\bigskip

We now introduce the space of functions of bounded mean oscillation. For every locally integrable function $f$ and every Calder\'on--Zygmund set $\tilde R$ we denote by $f_{\tilde R}$ the average of $f$ on $\tilde R$, i.e., $f_{\tilde R}=\frac{1}{\mu(\tilde R)}\int_{\tilde R}f\di\mu$.

\begin{definition}
Let $q\in [1,\infty)$. The space $\mathcal{B}\mathcal{M}\mathcal{O}_q(\mu)$ is the space of all functions in $L^q_{\rm{loc}}(\mu)$ such that
$$\sup_{\tilde R\in\mathcal R}\Big( \frac{1}{\mu(\tilde R)}\int_{\tilde R}|f-f_{\tilde R}|^q \di \mu \Big)^{1/q} <\infty\,.$$
The space $BMO_q(\mu)$ is the quotient of $\mathcal{B}\mathcal{M}\mathcal{O}_q(\mu)$ by constant functions. It is a Banach space endowed with the norm 
$$\|f\|_{BMO_q}=\sup\Big\{  \Big(  \frac{1}{\mu(\tilde R)}\int_{\tilde R}|f-f_{\tilde R}|^q \di\mu\Big)^{1/q} :~\tilde R   \in\mathcal R \Big\}\,.$$
\end{definition}
 We now prove the duality result between $BMO_1(\mu)$ and $H^1(\mu)$ and then show, as a consequence, that all $BMO_q(\mu)$ spaces coincide with $BMO_1(\mu)$, for $q\in (1,\infty)$.  
\begin{theorem}\label{t: duality}
The following hold:
\begin{itemize}
\item[(i)] for every function $f$ in $ BMO_1(\mu)$ there exists a bounded linear functional $L_f$ on $H^1(\mu)$ such that  
\begin{equation}\label{Lf}
L_f(g)=\int_{\mathcal V} f\,g\di\mu\qquad \forall g\in H^{1,\infty}_{\rm{fin}}(\mu)\,,
\end{equation}
and there exists $A>0$ such that $\|L_f\|_{(H^1)^*} \leq  A\,\|f\|_{BMO_1}$; 
\item[(ii)] for every bounded linear functional $L$ on $H^1(\mu)$ there exists a unique function $f\in BMO_1(\mu)$ such that $L=L_f$ and $\|f\|_{BMO_1}\leq C_2\,\|L \|_{(H^1)^*}$, where $C_2$ is the constant which appears in Proposition \ref{coincidono}(i). 
\end{itemize}
\end{theorem}
 \begin{proof}
We first prove (i) in the case when $f\in L^{\infty}(\mu)$. Let $g$ be a function in $H^1(\mu)$ and choose an atomic decomposition $g=\sum_j\lambda_ja_j$ such that $a_j$ are $(1,\infty)$-atoms supported in Calder\'on--Zygmund sets $\tilde R_j$. Since $f\in L^{\infty}(\mu)$ we have 
\begin{equation}\label{e: atdec}
\int f\,g\di\mu=\sum_j\lambda_j\int f\,a_j\di\mu\,.
\end{equation}
For every $j$
$$
\begin{aligned}
\Big| \int f\,a_j\di\mu   \Big|&=\Big| \int_{\tilde R_j} (f-f_{\tilde R_j}) \,a_j\di\mu   \Big|\\
&\leq  \int_{\tilde R_j} |f-f_{\tilde R_j}| \,|a_j|\di\mu\\
&\leq \mu(\tilde R_j)^{-1} \int_{\tilde R_j} |f-f_{\tilde R_j}| \di\mu\\
&\leq \|f\|_{BMO_1}\,.
\end{aligned}
$$
By \eqref{e: atdec} we deduce that
$$
\Big|\int f\,g\di\mu\Big|\leq \sum_j |\lambda_j| \|f\|_{BMO_1}\,.
$$
Taking the infimum over all atomic decompositions of $g$ we get 
\begin{equation}\label{e: flimitatabis}
\Big|\int f\,g\di\mu\Big|\leq \|g\|_{H^1} \|f\|_{BMO_1}\,.
\end{equation}
Let now $f\in BMO_1(\mu)$ be real-valued and define for every $k\in \mathbb N$ the function $f_k:\mathcal{V} \rightarrow \mathbb{R}$ by 
\begin{equation*}
f_k(x) = \begin{cases}
f(x) & \text{if } |f(x)| \leq k \\
k \frac{f(x)}{|f(x)|} & \text{if } |f(x)|>k.
\end{cases}
\end{equation*}
Then $\|f_k\|_{L^\infty} \leq k$ and $\|f_k\|_{BMO_1} \leq C\|f\|_{BMO_1}$. Moreover $|f_k-f|$ tends monotonically to zero when $k$ tends to $\infty$. Let $g\in H^{1,\infty}_{\rm{fin}}(\mu)$. By \eqref{e: flimitatabis} we deduce that
$$
\Big|\int f_k\,g\di\mu\Big|\leq \|g\|_{H^1}\|f_k\|_{BMO_1}\leq C\,\|g\|_{H^1}\|f\|_{BMO_1}\,.
$$
Since $f_kg$ tends to $fg$ everywhere, $g$ is compactly supported and $f$ is integrable on the support of $g$, by the dominated convergence theorem 
\begin{equation}\label{freal}
\Big|\int f\,g\di\mu\Big|=\lim_{k\rightarrow \infty} \Big|\int f_k\,g\di\mu\Big|\leq   C\,\|g\|_{H^1}\|f\|_{BMO_1} \qquad \forall g\in H^{1,\infty}_{\rm{fin}}(\mu)\,.
\end{equation}
Since $H^{1,\infty}_{\rm{fin}}(\mu)$ is dense in $H^1(\mu)$, it follows that the functional $L_{f}$ extends to a bounded functional on $H^1(\mu)$ and there exists a positive constant $A$ such that $\|L_f\|_{(H^1)^*}\leq A\|f\|_{BMO_1}$. 

If $f\in BMO_1(\mu)$ is complex-valued, then we apply \eqref{freal} to $\Re f$ and $\Im f$ and prove (i). 

\smallskip

 We now prove (ii). Let $\{\tilde Q_n\}$ be the sequence of Calder\'on--Zygmund sets constructed in Proposition \ref{l: ricoprimento} and, for any $n\in\mathbb N$, let $X_n^2$ and $X^2$ be the spaces introduced at the end of Section 2. Observe that
$H^{1,2}_{\rm{fin}}(\mu)$ and $X^2$ agree as vector spaces. For any $g\in X_n^2$ the function $\mu(\tilde Q_n)^{-1/2}  \|g\|_{L^2}^{-1} \,g$ is a (1,2)-atom, so that $g$ is in 
$H^{1,2}(\mu)$ and $
\|g\|_{H^{1,2}}\leq \mu (\tilde Q_n)^{1/2}  \|g\|_{L^2}$. Hence $X^2\subset H^{1,2}(\mu)$ and the inclusion is continuous.

Let $L$ be in $ (H^{1}(\mu))^*=(H^{1,2}(\mu))^*$. Hence $L$ lies in the
dual of $X^2$. Then by Proposition \ref{dualXp} there exists a function $f:\mathcal V\rightarrow \mathbb C$ such that
$$
L(g)=\int f\,g\di\mu \qquad \forall g\in X^2\,.
$$
We now show that $f\in BMO_2(\mu)$. Take a Calder\'on--Zygmund set $\tilde R$. For any $g\in X^2$ supported in $\tilde R$ the function $\mu(\tilde R)^{-1/2}  \|g\|_{L^2}^{-1}g$ is a (1,2)-atom. We then have 
$$
\Big|\int f g\di\mu\Big|=|L(g)|\leq \| L\|_{(H^{1,2})^*}\|g\|_{L^2}\mu(\tilde R)^{1/2}\,.
$$
This implies that $f-f_{\tilde R}$ is a function in $L^2_0(\tilde R)$ which represents the restriction of the bounded linear functional $L$ on $L^2_0(\tilde R)$. Hence 
$$
\|f-f_{\tilde R}\|_{L^2(\tilde R)}\leq \|L\|_{(L^2_0(\tilde R))^*}\leq \mu(\tilde R)^{1/2}\|L\|_{(H^{1,2})^*}.
$$
It follows that
$$
\frac{1}{\mu(\tilde R)}\int_{\tilde R}|f-f_{\tilde R}| \di\mu\leq \|L\|_{(H^{1,2})^*}\leq C_2\|L\|_{(H^{1})^*}\,,
$$ 
so that $f\in BMO_1(\mu)$ and $\|f\|_{BMO_1}\leq   C_2\| L\|_{(H^{1})^*}\,.$
 \end{proof}

\begin{corollary}\label{BMOq}
	 For every $q$ in $(1, \infty)$ the space $BMO_q(\mu)$ coincides with $BMO_1(\mu)$ and 
	 $$
\|f\|_{BMO_1}\leq \|f\|_{BMO_q} \leq A C_p\|f\|_{BMO_1} \qquad \forall f\in BMO_q(\mu)\,,
$$ 	 
where $p=q'$, $C_p$ is the constant which appears in Proposition \ref{coincidono}(i) and $A$ is the constant which appears in Theorem \ref{t: duality}(i). 
	 \end{corollary}
\begin{proof}
It follows from H\"older's inequality that 
$$
\|f\|_{BMO_1}\leq \|f\|_{BMO_q}\qquad \forall f\in BMO_q(\mu)\,.
$$ 
Take now $f\in BMO_1(\mu)$ and let $L_f$ be the functional on $H^1(\mu)$ such that 
$$
L_f(g)=\int f\,g\di\mu\qquad \forall g\in H^{1,\infty}_{\rm{fin}}(\mu)\,.
$$  
Let $q\in(1, \infty)$ and $p=q'$. 

Let $\{\tilde Q_n\}$ be the sequence of Calder\'on--Zygmund sets constructed in Proposition \ref{l: ricoprimento}. For any $n\in\mathbb N$, let $X_n^p$ and $X^p$ be the spaces introduced in Section 2. Observe that
$H^{1,p}_{\rm{fin}}(\mu)$ and $X^p$ agree as vector spaces. For any $g\in X_n^p$ the function $\mu(\tilde Q_n)^{-1+1/p}  \|g\|_{L^p}^{-1} \,g$ is a $(1,p)$-atom, so that $g$ is in 
$H^{1,p}(\mu)$ and $
\|g\|_{H^{1,p}}\leq \mu (\tilde Q_n)^{1-1/p}  \|g\|_{L^p}$. Hence $X^p \subset H^{1,p}(\mu)$ and the inclusion is continuous.

Hence $L_f$ lies in the
dual of $X^p$. Then by Proposition \ref{dualXp} there exists a function $F:\mathcal V\rightarrow \mathbb C $ such that
$$
L_f(g)=\int F\,g\di\mu \qquad \forall g\in X^p\,.
$$
We now show that $F\in BMO_q(\mu)$. Take a Calder\'on--Zygmund set $\tilde R$. For any $g\in X^p$ supported in $\tilde R$ the function $\mu(\tilde R)^{-1+1/p}  \|g\|_{L^p}^{-1}g$ is a $(1, p)$-atom. We then have 
$$
\Big|\int F\, g\di\mu\Big|=|L(g)|\leq \| L_f\|_{(H^{1,p})^*}\|g\|_{L^p}\,\mu(\tilde R)^{1-1/p}\,.
$$
This implies that $F-F_{\tilde R}$ is a function in $L^q_0(\tilde R)$ which represents the restriction of the bounded linear functional $L_f$ on $L^p_0(\tilde R)$. Hence 
$$
\|F-F_{\tilde R}\|_{L^q(\tilde R)}\leq \|L_f\|_{(L^p_0(\tilde R))^*}\leq \mu(\tilde R)^{1/q}\|L_f\|_{(H^{1,p})^*}.
$$
It follows that
$$
\Big(\frac{1}{\mu(\tilde R)}\int_{\tilde R}|F-F_{\tilde R}|^q \di\mu\Big)^{1/q}\leq \|L_f\|_{(H^{1,p})^*}\leq C_p\|L_f\|_{(H^{1})^*}\leq A C_p\|f\|_{BMO_1}\,,
$$ 
so that $\|F\|_{BMO_q}\leq  A C_p\|f\|_{BMO_1}$, where $C_p$ is the constant which appears in Proposition \ref{coincidono}(i).

Take any Calder\'on--Zygmund set $\tilde R$. Since by Proposition \ref{coincidono}(iii) $L^p_0(\tilde R)\subset H^{1,\infty}_{\rm{fin}}(\mu)$, we have that 
$$
L_f(g)=\int fg\di\mu=\int F\, g\di\mu\qquad \forall g\in L^p_0(\tilde R)\,.
$$
Hence there exists a positive constant $c_{\tilde R}$ such that $f=F-c_{\tilde R}$ on $\tilde R$, so that  
$$
\Big(\frac{1}{\mu(\tilde R)}\int_{\tilde R}|f-f_{\tilde R}|^q \di\mu\Big)^{1/q} =\Big(\frac{1}{\mu(\tilde R)}\int_{\tilde R}|F-F_{\tilde R}|^q \di\mu\Big)^{1/q}\,.
$$ 
In conclusion, $f
\in BMO_q(\mu)$ and $\|f\|_{BMO_q}=\|F\|_{BMO_q}\leq  A C_p\|f\|_{BMO_1} \,.$

\end{proof}

In the sequel we shall denote by $BMO(\mu)$ the space $BMO_1(\mu)$ endowed with the norm $\|\cdot\|_{BMO}=\|\cdot\|_{BMO_1}$.

 \smallskip
 
As a consequence of the duality result, by \cite[Theorem 2]{ATV1}, arguing exactly as in \cite[Section 5]{V} we can deduce the following real interpolation results involving Hardy and BMO spaces. 

\begin{corollary}
Suppose that $1\leq r_1<r<\infty$, $\frac{1}{r}=\frac{1-\theta}{r_1}$, $\theta\in (0,1)$. Then 
$$
[L^{r_1}(\mu),BMO(\mu)]_{\theta,q}=L^r(\mu)\,.
$$
Moreover, if $\frac{1}{r}={1-\theta}$, with $\theta\in (0,1)$, then
$$
[H^{1}(\mu),BMO(\mu)]_{\theta,r}=L^r(\mu)\,.
$$
\end{corollary} 
	
As a consequence of the duality result and of \cite[Theorem 3]{ATV1} we deduce that integral operators whose kernels satisfy a suitable integral H\"ormander condition are bounded from $L^{\infty}(\mu)$ to $BMO(\mu)$. 
%Note that the integral H\"ormander condition which we require below is weaker than the integral conditions in the hypothesis of \cite[Theorem 1.2]{HS}. 

\begin{corollary}\label{TeolimH1L1}
Let $T$ be a linear operator which is bounded on $L^2(\mu)$ and admits a locally integrable kernel $K$ off the diagonal that satisfies the condition 
$$
\begin{aligned}%\label{stimaH}
\sup_{\tilde R}\sup_{y,\,z\in \tilde R}   \int_{(\tilde R^*)^c}|K(y,x)-K(z,x)| \,\di\mu (x) &<\infty\,,
\end{aligned}
$$ 
where the supremum is taken over all Calder\'on-Zygmund sets $\tilde R$ and $\tilde R^*$ is defined as in Definition~\ref{tildeQ*}. Then $T$ extends to a bounded operator from $L^{\infty}(\mu)$ to $BMO(\mu)$.

\end{corollary}

	\section{The sharp maximal function}

In this section we introduce the sharp maximal function associated with the family of Calder\'on--Zygmund sets and investigate some of its properties. This sharp maximal function is related with the definition of the $BMO$-space and might be useful to study interpolation properties of such space. 
	
	\begin{definition}
Let $q\in [1,\infty)$. For every function $f$ in $L^q_{\rm{loc}}(\mu)$ its sharp maximal function $f^{\sharp,q}$ is defined by 
$$f^{\sharp,q}(x)=\sup_{\tilde R\in\mathcal R, x\in\tilde R}  \Big(  \frac{1}{\mu(\tilde R)}\int_{\tilde R}|f-f_{\tilde R}|^q \di\mu\Big)^{1/q}  \qquad \forall x\in \mathcal V\,.$$
%We shall simply denote by $f^{\sharp}$ the sharp maximal function $f^{\sharp,1}$. 
\end{definition}

Notice that $
\|f^{\sharp,q}\|_{L^\infty}=\|f\|_{BMO_q}
$ for every function $f\in BMO(\mu)$.

\begin{proposition} \label{properties-sharp}
Let $q\in [1,\infty)$ and $f,g\in BMO(\mu)$. The following hold:
\begin{itemize}
\item[(i)] 
$$\frac{1}{2}f^{\sharp,q}(x)\leq  \sup_{\tilde R\in\mathcal R, x\in\tilde R} \inf_{c\in\mathbb C}  \Big(  \frac{1}{\mu(\tilde R)}\int_{\tilde R}|f-c|^q \di\mu\Big)^{1/q} \leq f^{\sharp,q}(x)\qquad \forall x\in\mathcal V\,;$$
%Then $f^{\sharp,q}(x)\leq 2A^q_x$ and $\|f\|_{BMO_q}\leq 2 \sup_{x\in\mathcal V} A^q_x$;
\item[(ii)] \,for every $x\in\mathcal V$, $|f|^{\sharp,q}(x)\leq 2f^{\sharp, q}(x)$. Then $|f|\in BMO(\mu)$ and $$\| \,|f|\,\|_{BMO_q}\leq 2\|f\|_{BMO_q}\,;$$
\item[(iii)] $(f+g)^{\sharp, q}(x)\leq f^{\sharp,q}(x)+g^{\sharp,q}(x)$ for every $x\in\mathcal V$;
\item[(iv)] there exists a positive constant $C$ such that if $f$ and $g$ have real values, then for every $x\in\mathcal V$, 
$$
[\max(f,g)]^{\sharp,q}(x)\leq C\,\big( |f|^{\sharp,q}(x)+ |g|^{\sharp,q}(x)\big),\qquad [\min(f,g)]^{\sharp,q}(x)\leq C\, \big(|f|^{\sharp,q}(x)+ |g|^{\sharp,q}(x)\big)\,.
$$
\end{itemize}
\end{proposition}
\begin{proof}
We first prove (i). Given $x\in\mathcal V$ and $\tilde R\in\mathcal R$ which contains $x$ for every $\varepsilon>0$ we choose a constant $c_{\tilde R}$ such that 
$$
   \Big(  \frac{1}{\mu(\tilde R)}\int_{\tilde R}|f-c_{\tilde R}|^q \di\mu\Big)^{1/q}
\leq \inf_{c\in\mathbb C}  \Big(  \frac{1}{\mu(\tilde R)}\int_{\tilde R}|f-c|^q \di\mu\Big)^{1/q}+\varepsilon 
$$
Then, by applying H\"older's inequality, we obtain 
$$
\begin{aligned}
\Big( \frac{1}{\mu(\tilde R)}\int_{\tilde R} |f-f_{\tilde R}|^q\di\mu\Big)^{1/q}&\leq \frac{1}{\mu(\tilde R)^{1/q}} \|f-c_{\tilde R}\|_{L^q(\tilde R)}+ \frac{1}{\mu(\tilde R)^{1/q}} \|f_{\tilde R}-c_{\tilde R}\|_{L^q(\tilde R)} \\
&\leq  2\inf_{c\in\mathbb C}  \Big(  \frac{1}{\mu(\tilde R)}\int_{\tilde R}|f-c|^q \di\mu\Big)^{1/q}+2\varepsilon \,.
%&\leq A^q_x+\Big( \frac{1}{\mu(\tilde R)}\int_{\tilde R} \frac{1}{\mu(\tilde R)^q} \Big[\int_{\tilde R}|f-c_{\tilde R}|^q\di\mu\Big]\mu(\tilde R)^{q/q'}\di\mu  \Big)^{1/q}
% 2A^q_x\,.
\end{aligned}
$$
Since $\varepsilon>0$ is arbitrary and taking the supremum over all sets $\tilde R\in\mathcal R$ we obtain the first inequality in (i); the second one is immediate. 

Let us now prove (ii). For every $\tilde R\in\mathcal R$, we have
$$
\begin{aligned}
|f|^{\sharp,q}(x)&\leq 2\sup_{\tilde R\in\mathcal R, x\in\tilde R}\inf_{c\in\mathbb C}\Big( \frac{1}{\mu(\tilde R)}\int_{\tilde R} \big| |f|-c\big|^q\di\mu\Big)^{1/q}\\
&\leq 2\sup_{\tilde R\in\mathcal R, x\in\tilde R} \Big( \frac{1}{\mu(\tilde R)}\int_{\tilde R} \big| |f|-|f_{\tilde R}|\big| ^q\di\mu\Big)^{1/q}\\
&\leq 2\sup_{\tilde R\in\mathcal R, x\in\tilde R} \Big( \frac{1}{\mu(\tilde R)}\int_{\tilde R}  |f-f_{\tilde R}|^q\di\mu\Big)^{1/q}\\
&= 2f^{\sharp,q}(x)\,.
\end{aligned}
$$ 
Property (iii) is an immediate consequence of the definition of the sharp maximal function. 

Property (iv) follows from (ii) and the fact that 
$$
\max(f,g)=\frac{f+g+|f-g|}{2}\qquad{\rm{and}}\qquad  \min(f,g)=\frac{f+g-|f-g|}{2}\,.
$$
\end{proof}
Notice that, as a consequence of Proposition \ref{properties-sharp} (iv), the set of real valued functions in $BMO(\mu)$ is a lattice. 

In the following result, we explain how the duality of $H^1(\mu)$ with $BMO(\mu)$ can be quantitatively expressed in terms of the sharp maximal function. 
\begin{proposition}\label{intfg}
Let $q\in (1,\infty)$. There exists a positive constant $C$ such that for every $f\in L^{\infty}(\mu)$ and $g\in H^1(\mu)\cap L^q(\mu)$ 
$$
\Big|\int f\,g\,\di\mu\Big|\leq C\,  \int f^{\sharp,q'} (M(|g|^q))^{1/q}\di\mu\,,
$$
where 
$$
M(|g|^q)(x)=\sup_{x\in R}\frac{1}{\mu(R)}\int_R|g|^q\di\mu\qquad \forall x\in\mathcal V\,,
$$
where the supremum is taken over all admissible trapezoids that contain $x$. 
\end{proposition}
\begin{proof}
Take $g\in H^1(\mu)\cap L^q(\mu)$. For every $j\in\mathbb Z$ arguing as in \cite[Lemma 1]{ATV1} we can construct a family of disjoint trapezoids $R_{k}^j$, a function $g^j$ and functions $b^j_k$ such that
\begin{itemize}
\item $\bigcup_kR^j_k\subset \Omega_j=\{x\in\mathcal V: M(|g|^q)>2^{jq}\}\subset \bigcup_k\tilde R^j_k$;
\item $g=g^j+\sum_kb^j_k=g^j+b^j$;
\item $|g^j|\leq C2^j$;
\item $b^j_k$ is supported in $\tilde R^j_k$, has vanishing integral and $\|b^j_k\|_{L^q}\leq C2^j\mu(\tilde R^j_k)^{1/q}$.
\end{itemize}
These facts implies that $g^j$ tends to $0$ uniformly when $j\rightarrow -\infty$ and $\|b^j\|_{H^1}\leq C 2^{j(1-q)}\|g\|^q_{L^q}$, hence $b^j$ tends to $0$ in $H^1(\mu)$ when $j\rightarrow +\infty$. Thus
\begin{equation}\label{g=lim}
g=\lim_{N\rightarrow +\infty}\sum_{j=-\infty}^N(g^{j+1}-g^j)=\lim_{N\rightarrow +\infty}\sum_{j=-\infty}^N(b^{j}-b^{j+1})\,.
%=\sum_j\sum_k b^j_k-\sum_j\sum_{\ell}b^{j+1}_\ell\,.
\end{equation}
Notice that we can write $b^j_k=\lambda^j_k a^j_k$, where $\lambda^j_k=C2^j\mu(\tilde R^j_k)$ and $a^j_k$ is a $(1,q)$-atom supported in $\tilde R^j_k$.

Take now $f\in L^{\infty}(\mu)$. We have 
$$
\Big|\int f\,a^j_k\di\mu\Big|\leq \Big|\int [f-f_{\tilde R^j_k}]\,a^j_k\di\mu\Big|\leq \|a^j_k\|_{L^q} \|f-f_{\tilde R^j_k}\|_{L^{q'}(\tilde R^j_k)}\leq    \mu(\tilde R^j_k)^{-1+1/q}\|f-f_{\tilde R^j_k}\|_{L^{q'}(\tilde R^j_k)}\leq  f^{\sharp,q'}(x)\,,
$$
for every $x\in \tilde R^j_k$. Hence
$$
\Big|\int f\,a^j_k\di\mu\Big|\leq \frac{1}{\mu(R^j_k)}\int_{R^j_k} f^{\sharp,q'}\di\mu\,.
$$
Using \eqref{g=lim} and Proposition \ref{mutildeR}, it follows that 
$$
\begin{aligned}
\Big|\int f\,g\di\mu\Big|&\leq \sum_{j,k}\lambda^j_{k} \Big|\int f\,a^j_k\di\mu\Big| +\sum_{j,\ell}\lambda^{j+1}_{\ell} \Big|\int f\,a^{j+1}_{\ell}\di\mu\Big| \\
&\leq C \sum_{j,k}2^j \mu(\tilde R^j_k) \frac{1}{\mu(R^j_k)}\int_{R^j_k} f^{\sharp,q'}\di\mu+C \sum_{j,\ell}2^{j+1} \mu(\tilde R^{j+1}_\ell) \frac{1}{\mu(R^{j+1}_\ell)}\int_{R^{j+1}_\ell} f^{\sharp,q'}\di\mu\\
&\leq C\sum_j \int_{\bigcup_k R^j_k} 2^j\,f^{\sharp,q'}\di\mu +C\sum_j \int_{\bigcup_\ell R^{j+1}_\ell} 2^{j+1}\,f^{\sharp,q'}\di\mu \\
&\leq C \sum_j \int_{ \Omega_j } 2^j\,f^{\sharp,q'}\di\mu +C\sum_j \int_{ \Omega_{j+1}   } 2^{j+1}\,f^{\sharp,q'}\di\mu \\
&\leq C\int_{\mathcal V} f^{\sharp,q'}(x) \sum_{2^j<(M|g|^q)^{1/q}(x)} 2^j\di\mu(x)+ C\int_{\mathcal V} f^{\sharp,q'}(x) \sum_{2^{j+1}<(M|g|^q)^{1/q}(x)} 2^{j+1} \di\mu(x)\\
&\leq C\int_{\mathcal V} f^{\sharp,q'}(x) (M|g|^q)^{1/q}(x) \di\mu(x)\,,
\end{aligned}
 $$
as required.
\end{proof}

The following theorem can be thought as a weak version of the classical $L^p$-inequality involving the sharp maximal function in the Euclidean setting (see \cite[Theorem 2, \S 2, Ch. IV]{S}). 
\begin{theorem}\label{t: normaLpfsharp}
Let $p\in (1,\infty)$ and $p_0\in (1,p)$. There exists a positive constant $C$ such that 
\begin{equation}
\label{weaksharp}
\|f\|_{L^p}\leq C\, \|f^{\sharp,p_0}\|_{L^p}\qquad \forall f\in L^{p_0}(\mu)\,.
\end{equation}
\end{theorem}
\begin{proof}
Given $f\in L^{p_0}(\mu)$ real-valued we define for every $k\in\mathbb N$
$$
f_k(x) = \begin{cases}
f(x) & \text{if } |f(x)| \leq k \\ 
k \frac{f(x)}{|f(x)|} & \text{if } |f(x)|>k.
\end{cases}
$$ 
Then $f_k$ converges to $f$ in $L^{p_0}(\mu)$ when $k\rightarrow +\infty$, $f_k\in L^{\infty}(\mu)$ and by Proposition \ref{properties-sharp} (iv) 
\begin{equation}\label{pointwiseboundsharp}
f_k^{\sharp,p_0}\leq C\, f^{\sharp,p_0}\,.
\end{equation}
Denote by $q$ and $q_0$ the conjugate exponents of $p$ and $p_0$, respectively. Take a function $g\in L^q(\mu)\cap L^{q_0}(\mu)$ with $\|g\|_{L^q}\leq 1$. For every $k\in\mathbb N$ define 
$$
\tilde g_k=g\chi_{\tilde Q_k}\qquad{\rm{and}}\qquad  I_k=\int \tilde g_k\di\mu\,,
$$
where $\tilde Q_k$ are the sets introduced in Proposition \ref{l: ricoprimento}. We then define
$$
g_k=\tilde g_k-I_k\,\mu(\tilde Q_{2^k})^{-1}\chi_{\tilde Q_{2^k}}=\tilde g_k-r_k\,.
$$
The function $g_k$ is in $L^q(\mu)$, it has vanishing integral and is supported in $\tilde Q_{2^k}$; hence $g_k$ lies in $H^1(\mu)$. Moreover, $g_k$ tends to $g$ pointwise for $k\rightarrow +\infty$ and $|g_k|\leq |g|+r$, where $r=\sum_k |r_k|$. We have that 
$$
\|r_k\|_{L^q}\leq \mu(\tilde Q_k)^{1/{q'}}\mu(\tilde Q_{2^k})^{-1+1/q}=\Big(\frac{q^k(k+1)}{q^{2^k}(2^k+1)}\Big)^{1/{q'}}\,,
$$
and then $r\in L^q(\mu)$. Hence $g_k$ tends to $g$ also in $L^q(\mu)$ for $k\rightarrow +\infty$.

By applying Proposition \ref{intfg} to $f_k$ and $g_k$ and H\"older's inequality we get
$$
\Big|\int f_k\,g_k\di\mu\Big|\leq \int f_k^{\sharp,p_0} (M(|g_k|^{q_0}))^{1/{q_0}}\di\mu\leq \|f_k^{\sharp,p_0} \|_{L^p} \,\|(M(|g_k|^{q_0}))^{1/{q_0}}\|_{L^q}\,.
$$
Taking the limit for $k\rightarrow +\infty$ on the left-hand side and by applying \eqref{pointwiseboundsharp} and the boundedness of the Hardy--Littlewood maximal function on $L^{q/q_0}(\mu)$ (see \cite{ATV1}) we deduce that
\begin{equation}\label{estimate-freal}
\Big|\int f \,g\di\mu\Big|\leq C\, \|f^{\sharp,p_0} \|_{L^p} \,\|g_k \|_{L^q}\leq C\, \|f^{\sharp,p_0} \|_{L^p} \,\|g \|_{L^q}\leq C\,\|f^{\sharp,p_0} \|_{L^p}\,.
\end{equation}
Since the previous inequality holds for every $g\in L^q(\mu)\cap L^{q_0}(\mu)$ with $\|g\|_{L^q}\leq 1$ with constants independent of $g$, we deduce that $\|f\|_{L^p}\leq C\, \|f^{\sharp,p_0}\|_{L^p}$. 

The case when $f$ is complex-valued follows by applying \ref{estimate-freal} to $\Re f$ and $\Im f$ and arguing as above. 
\end{proof}
	
\begin{remark}
Proposition \ref{intfg} and Theorem \ref{t: normaLpfsharp} are inspired by similar results involving the sharp maximal function in the Euclidean setting (see \cite[\S 2, Ch. IV]{S}). More precisely, Proposition \ref{intfg} differs from \cite[Formula (16) Ch. IV]{S}  because we require an extra integrability condition on the function $g\in H^1(\mu)$ and the maximal function involved here is a variant of the Hardy--Littlewood maximal function. This is due to the fact that a maximal characterization of the Hardy space $H^1(\mu)$ is not available in our setting.  

The inequality \eqref{weaksharp} is a weak version of the inequality 
$$
\|f\|_{L^p}\leq C\, \|f^{\sharp,1}\|_{L^p}\,,$$
which is still unknown in the setting of the present paper. The proof of such inequality would probably require a distributional inequality involving both the Hardy--Littlewood and the sharp maximal functions (or a dyadic version of them) that is still work in progress. %Such an improvement should also imply a complex interpolation result for the $BMO$-space. 

\end{remark}

%\end{remark}	

	\bigskip
	
	{\bf{Acknowledgments.}} 	
	Work partially supported by the MIUR project ``Dipartimenti di Eccellenza 2018-2022" (CUP E11G18000350001). 
	
	The authors are members of the Gruppo Nazionale per l'Analisi Matema\-tica, la Probabilit\`a e le loro Applicazioni (GNAMPA) of the Istituto Nazionale di Alta Matematica (INdAM).

\end{document}